\documentclass[11pt, oneside]{amsart}   	
\usepackage{amsmath,amsthm,amscd,amsfonts, amssymb, mathrsfs}
\usepackage{geometry}                		
\usepackage{graphicx}				
\usepackage{amssymb}
\newcommand{\sect}[1]{\section{#1}\setcounter{equation}{0}}
\newcommand{\subsect}[1]{\subsection{#1}}

\font\mbn=msbm10 scaled \magstep1
\font\mbs=msbm7 scaled \magstep1
\font\mbss=msbm5 scaled \magstep1
\newfam\mbff
\textfont\mbff=\mbn
\scriptfont\mbff=\mbs
\scriptscriptfont\mbff=\mbss   

\makeatletter
\@namedef{subjclassname@2020}{\textup{2020} Mathematics Subject Classification}
\makeatother

\newcommand{\Di}      {\mathbb{D}}

\newcommand{\N}       { \mathbb{N}}
  
\newcommand\Co           {{\mathbb C}}

\newtheorem{Th}{Theorem}[section]
\newtheorem{Lm}[Th]{Lemma}
\newtheorem{C}[Th]{Corollary}

\newtheorem{Prop}[Th]{Proposition}
\newtheorem{R}[Th]{Remark}

\newtheorem*{Theorem}{Theorem}
\newtheorem*{Con}{Conjecture}
\newtheorem*{Th A}{Theorem A}
\newtheorem*{Th B}{Theorem B}
\begin{document}

\title[Runge-type Approximation Theorem for Banach-valued $ H^\infty$ Functions]{Runge-Type Approximation Theorem for Banach-valued ${\mathbf H^\infty}$ Functions on a Polydisk}
\author{Alexander Brudnyi }
\address{Department of Mathematics and Statistics\newline
\hspace*{1em} University of Calgary\newline
\hspace*{1em} Calgary, Alberta, Canada\newline
\hspace*{1em} T2N 1N4}
\email{abrudnyi@ucalgary.ca}

\keywords{Banach-valued bounded holomorphic function, maximal ideal space, Runge-type approximation, interpolating Blashcke product}
\subjclass[2020]{Primary 32E30. Secondary 30H05.}

\thanks{Research is supported in part by NSERC}

\begin{abstract}
Let $\Di^n\subset\Co^n$ be the open unit polydisk, $K\subset\Di^n$ be an $n$-ary Cartesian product of planar sets, and $\widehat U\subset \mathfrak M^n$ be an open neighbourhood of the closure $\bar K$ of $K$ in $\mathfrak M^n$, where $\mathfrak M$ is the maximal ideal space of the algebra $H^\infty$ of bounded holomorphic functions on $\Di$. Let $X$ be a complex Banach space and $H^\infty(V,X)$ be the space of bounded $X$-valued holomorphic functions on an open set $V\subset\Di^n$.
We prove that any  $f\in H^\infty(U,X)$, where  $U=\widehat U\cap\Di^n$, can be uniformly approximated on $K$ by ratios $h/b$, where $h\in H^\infty(\Di^n,X)$ and $b$ is the product of interpolating Blaschke products such that $\inf_K |b|>0$. Moreover,
if $\bar K$ is contained in a compact holomorphically convex subset of $\widehat U$, then $h/b$ above can be replaced by $h$ for any $f$. The results follow from a new constructive Runge-type approximation theorem for Banach-valued holomorphic functions on open subsets of $\Di$ and extend the fundamental results of Su\'{a}rez on Runge-type approximation  for analytic germs on compact subsets of $\mathfrak M$. They can also be applied to the long-standing corona problem which asks whether $\Di^n$ is dense in the maximal ideal space of $H^\infty(\Di^n)$ for all $n\ge 2$.   
\end{abstract}

\date{}

\maketitle

\sect{Introduction}
Let $H^\infty$ be the Banach algebra of bounded holomorphic functions on the open unit disk $\Di\subset\Co$ equipped with pointwise multiplication and supremum norm, 
and let $\mathfrak M$ be its maximal ideal space.
Recall that for a commutative unital complex Banach algebra $A$,  the maximal ideal space $\mathfrak M(A)\subset A^\ast$  
 is the set of nonzero homomorphisms $A \!\rightarrow\! \Co$ endowed with the Gelfand topology, the weak-$\ast$ topology of  $A^\ast$. It is a compact Hausdorff space contained in the unit sphere of $A^\ast$. The {\em Gelfand transform} defined by $\hat{a}(\varphi):=\varphi(a)$ for $a\in A$ and $\varphi \in \mathfrak M(A)$ is a nonincreasing-norm morphism from $A$ into the Banach algebra $C(\mathfrak M (A))$ of complex-valued continuous functions on $\mathfrak M(A)$. 

In the case of $H^\infty$, the {\em Gelfand transform}   $\hat{\,}: H^\infty\to C(\mathfrak M)$ is an isometry and
 the map $\iota:\Di\hookrightarrow \mathfrak M$ taking $z\in \Di$ to the evaluation homomorphism $f\mapsto f(z)$, $f\in H^\infty$, is an embedding with dense image by the celebrated Carleson corona theorem \cite{C}. In the sequel, we identify $\Di$ with $\iota(\Di)$ and regard $\Di$ as an open dense subset of $\mathfrak M$.
 
 In \cite{S1}, Su\'{a}rez proved Runge-type approximation theorems for
 analytic germs on compact subsets of $\mathfrak M$, analogous to the classical  rational and polynomial  Runge approximation theorems for compact subsets of $\Co$. Specifically, he proved the following results.
 
 A compact set $K\subset\mathfrak M$ is said to be {\em holomorpically convex} if for each $x\not\in K$ there is a function $f\in H^\infty$ such that
 \[
 \sup_K |\hat f|<|\hat f(x)|.
 \]
 \begin{Theorem}[\mbox{\cite[Thm.\,3.3(I),\,Cor.\,2.6]{S1}}]
 Let $K\subset\mathfrak M$ be a compact set and let $\varepsilon>0$. Suppose that $f$ is a continuous function in an open neighbourhood $U$ of $K$  holomorphic in $U\cap \Di$.
 \begin{itemize}
 \item[(i)]
 There exist $h\in H^\infty$ and an interpolating Blaschke product $b$ such that $\hat b$ has no zeros in $K$ and
 \[
 \sup_K| f-\hat h/\hat b|<\varepsilon.
 \]
 \item[(ii)]
 If $K$ is holomorphically convex, then there exists  $h\in H^\infty$ such that
 \[
 \sup_K| f-\hat h|<\varepsilon.
 \]
 \end{itemize}
 \end{Theorem}
 For example, using part (i) of the theorem, one can give an alternative proof of Su\'{a}rez's theorem \cite[Thm.\,2.4]{S0} which states that the algebra $H^\infty$ is {\em separating}, i.e., that if $K$ is a closed subset of $\mathfrak M$ and $x\in\mathfrak M\setminus K$, then there exists a function $f\in H^\infty$ such that $\hat f(x)\notin \hat f(K)$, see section~2.2 below.\smallskip
 
 In this paper we prove the following extension of Su\'{a}rez's results to the case of Banach-valued holomorphic functions on open subsets of the unit polydisk $\Di^n\subset\Co^n$. 

Let $X$ be a complex Banach space with norm
 $\lVert\cdot\rVert_X$.  For an open set $V\subset\Co^n$, we denote by  $H^\infty(V,X)$ the Banach space of $X$-valued holomorphic functions $f$ on $V$  with norm $\|f\|_{H ^\infty(V,X)}:=\sup_{z\in V}\|f(z)\|_X$, and define $H^\infty(V):=H^\infty(V,\Co)$.  \begin{Th}\label{cor1.3}
Let $K\subset \Di^n$ be an $n$-ary Cartesian product of planar sets, $\widehat U\subset \mathfrak M^n$ be an open neighbourhood of the closure $\bar K$ of $K$ in $\mathfrak M^n$, and let $\varepsilon>0$. Suppose $f\in H^\infty(U,X)$, where $U:=\widehat U\cap\Di^n$.\footnote{Since $\Di^n$ is an open  dense subset of $\mathfrak M^n$, $U$ is an open subset of $\Di^n$ dense in $\widehat U$.}
 \begin{itemize}
 \item[(i)]
 There exist $h\in H^\infty(\Di^n,X)$, $b\in H^\infty(\Di^n)$ such that $b(z)=b_1(z_1)\cdots b_n(z_n)$, $z=(z_1,\dots, z_n)\in\Di^n$, where every $b_i\in H^\infty$ is an interpolating Blaschke product, and  $\inf_K |b|>0$, and  
 \[
 \sup_{z\in K}\| f(z)- h(z)/ b(z)\|_X<\varepsilon.
 \]
 \item[(ii)]
 If $\bar K$ is contained in a compact subset of $\widehat U$ which  is an $n$-ary Cartesian product of holomorphically convex subsets of $\mathfrak M$, then there exists  $h\in H^\infty(\Di^n,X)$ such that
 \[
 \sup_{z\in K}\| f(z)- h(z)\|_X<\varepsilon.
 \]
 \item[(iii)]
 If $f$ admits a continuous extension to $\widehat U$, then the functions $h\in H^\infty(\Di^n,X)$ in (i) an (ii) can be chosen so that they admit continuous extensions to $\mathfrak M^n$.
 \end{itemize}
 \end{Th}
 Let $(H^\infty)^{\widehat{\otimes}_\varepsilon^n}\subset H^\infty(\Di^n)$ be the 
 $n$-fold injective tensor product of $H^\infty$, i.e.,
 the closed subalgebra of $H^\infty(\Di^n)$ generated by univariate $H^\infty$ functions in coordinates on $\Di^n$. Let $A_n\subset C(\Di^n)$ be the uniform subalgebra generated by the algebra $(H^\infty)^{\widehat{\otimes}_\varepsilon^n}$ and its complex conjugate. Clearly, functions in $A_n$ admit continuous extensions to $\mathfrak M^n$
 and the algebra $\hat A_n$ of the  extended functions separates points of  $\mathfrak M^n$. Hence, by the Stone-Weierstrass theorem, $\hat A_n=C(\mathfrak M^n)$. This implies that 
 if $\widehat U$ is an open subset of $\mathfrak M^n$,  then an $X$-valued holomorphic function on
$U=\widehat U\cap\Di^n$ admits a continuous extension to $\widehat U$ if and only if, 
on every open subset $V$ of $U$ that is relatively compact in $\widehat U$ (written $V\Subset\widehat U$), it can be represented as the uniform limit of a sequence of functions from $X\otimes A_n$.
In turn, in the one-dimensional case, an $X$-valued holomorphic function $f$ on $U$ admits a continuous extension to $\widehat U\subset\mathfrak M$ if and only if,  $f(V)\Subset X$ for every subset $V\Subset\widehat U$ of $U$. This follows from another result of Su\'{a}rez \cite[Thm.\,3.2]{S0}, see \cite[Prop.\,1.3]{Br1}. Theorem \ref{cor1.3} gives the following extension of this result.

 Let ${\rm cl}(\Di^n)$ be the closure of $\Di^n$ (viewed as the set of evaluation homomorphisms at points of $\Di^n$) in $\mathfrak M(H^\infty(\Di^n))$. Then the transpose of the Banach algebra monomorphism $(H^\infty)^{\widehat{\otimes}_\varepsilon^n}\hookrightarrow H^\infty(\Di^n)$ induces a continuous surjection $\pi_n:{\rm cl}(\Di^n)\to\mathfrak M^n$ which is identity on $\Di^n$ and maps ${\rm cl}(\Di^n)\setminus\Di^n$ onto $\mathfrak M^n\setminus\Di^n$.
 \begin{C}\label{cor1.2}
 Let $\widehat U\subset\mathfrak M^n$ be an open set. An  $X$-valued holomorphic function $f$ on $U=\widehat U\cap\Di^n$ admits a continuous extension to the open set $\pi_n^{-1}(\widehat U)\subset {\rm cl}(\Di^n)$ if and only if $f(V)\Subset X$ for  every subset $V\Subset\widehat U$ of $U$.
 \end{C}
 \begin{R}\label{rem1.3}
 {\rm (1) The corollary implies that every function $f\in H^\infty(U,X)$ (with $U$ as above) with a relatively  compact image admits a continuous extension to $\pi_n^{-1}(\widehat U)$. 
 
 \noindent (2)  The map $\pi_n:{\rm cl}(\Di^n)\to\mathfrak M^n$ is not one-to-one over points outside $\Di^n$, see, e.g., \cite{Cu};
for instance, some fibres of $\pi_n$ contain subsets homeomorphic to $\Di^n$. Some results on the structure of $\pi_n$ will be presented in a forthcoming paper.

\noindent (3)  The long-standing corona problem for $H^\infty(\Di^n)$, $n\ge 2$, asks  whether ${\rm cl}(\Di^n)$ coincides with $\mathfrak M(H^\infty(\Di^n))$. There are currently no significant developments in this area.  Some applications of Theorem \ref{cor1.3} related to this problem will be published elsewhere.
 }
 \end{R}
 According to part (2) of the remark, there are many open subsets of ${\rm cl}(\Di^n)$ which are not preimages of open subsets of $\mathfrak M^n$ under $\pi_n$. So one can ask whether the analog of Corollary \ref{cor1.2} holds for them. In particular, the following conjecture seems plausible.
 \begin{Con}
 Let $\widehat U\subset {\rm cl}(\Di^n)$ be an open set.
 An  $X$-valued holomorphic function $f$ on $U=\widehat U\cap\Di^n$ admits a continuous extension to $\widehat U$ if and only if $f(V)\Subset X$ for every subset $V\Subset\widehat U$ of $U$.
 \end{Con}
Theorem \ref{cor1.3} follows from a new constructive Runge-type approximation theorem for Banach-valued holomorphic functions on open subsets of $\Di$ presented in the next section.
\sect{Main Result}
Recall that a sequence $\{z_n\}\subset\Di$ is said to be {\em interpolating} (for $H^\infty$) if every interpolation problem
\begin{equation}\label{eq1.1}
g(z_n)=a_n,\quad n\ge 1,
\end{equation}
with a bounded data $\{a_n\}\subset\Co$ has a solution $g\in H^\infty$.

Clearly, every finite subset of $\Di$ is an interpolating sequence. In general, by the Carleson theorem, see, e.g., \cite[Ch.\,VII,\,Thm.\,1.1]{Ga}, a sequence $\zeta=\{z_n\}\subset\Di$ is interpolating  if and only if  the characteristic 
\begin{equation}\label{eq1.2}
\delta(\zeta):=\inf_{k}\prod_{j,\,j\ne k}\rho(z_j,z_k)>0,\!\footnote{Here $\delta(\zeta)=1$ if $\zeta$ consists of one element.}
\end{equation}
where
\begin{equation}\label{eq1.3}
\rho(z,w):=\left|\frac{z-w}{1-\bar w z}\right|,\qquad z,w\in\mathbb D,
\end{equation}
is the pseudohyperbolic metric on $\mathbb D$.

If $\zeta=\{z_n\}\subset\mathbb{D}$ is an interpolating sequence, then the Blaschke product $B_\zeta$ having simple zeros at points of $\zeta$,
\begin{equation}\label{eq1.4}
\begin{split}
B_\zeta(z) &  :=\prod_{n} \mu(z_n)\frac{z-z_n}{1-\bar{z_n}z},\quad z\in\mathbb{D}, \\
&  \qquad \mu(z):=-\frac{z}{|z|}\quad {\rm if}\quad z\ne 0 \quad {\rm and}\quad  \mu(0):=1,
\end{split}
\end{equation}
 is said to be {\em interpolating}.  Note that
 \[
 \delta(\zeta)=\inf_n(1-|z_n|^2)B_\zeta'(z_n).
 \]
 
 Let $X$ be a complex Banach space. 
 For an open set $V\subset\Di$ 
we denote by $H_{\rm comp}^\infty(V,X)$ the closed subspace of $H^\infty(V,X)$  functions with relatively compact images. If $V=\Di\cap\widehat V$ for an open set $\widehat V\subset\mathfrak M$, then every function  $f\in H_{\rm comp }^\infty(V,X)$ extends to a continuous function with a relatively compact image $\hat f\in C(\widehat V,X)$, see \cite[Thm.\,3.2]{S0} and \cite[Prop.\,1.3]{Br1}. It follows that every function $f\in H^\infty(V,X)$ extends to a weak-$*$ continuous function  $\hat f$ on $\widehat V$ with values in the bidual $X^{**}$ of $X$ such that $\sup_{x\in\widehat V}\|\hat f(x)\|_{X^{**}}=\|f\|_{H^\infty(V,X)}$. Let $\Gamma\subset\mathfrak M$ be the Shilov boundary of $H^\infty$ (i.e., $\Gamma$ is the smallest closed subset of $\mathfrak M$ such that $\|\hat f\|_{C(\Gamma)}=\|f\|_{H^\infty}$ for all $f\in H^\infty$).
Then we define
\begin{equation}\label{eq1.5}
|f|_{H^\infty(V,X)|_\Gamma}:=\sup_{x\in \widehat V\cap\Gamma}\|\hat f(x)\|_{X^{**}}.\footnote{Here $|f|_{H^\infty(V,X)|_\Gamma}=0$ if $\widehat V\cap\Gamma=\emptyset$.}
\end{equation}
Clearly, $|\cdot|_{H^\infty(V,X)|_\Gamma}$ is a seminorm on $H^\infty(V,X)$.

In the sequel, we use the following notation.\smallskip
 \[
 \begin{split}
D(x,r):=&\{z\in\Co\, :\, \rho(z,x)<r\},\quad \Di_r:=\{z\in\Co\, :\, |z|<r\},\\
\rho(S_1,S_2):=&\inf\{\rho(z_1,z_2)\, :\, z_1\in S_1,\, z_2\in S_2\},\quad S_1,S_2\subset\Di.
\end{split}
 \]
 We denote by $\bar{S}$ the closure of $S\subset\mathfrak M$.
For an interpolating sequence $\zeta\subset\Di$  by $B_\zeta$ we denote the interpolating Blaschke product having simple zeros at points of $\zeta$.
We write $C=C(\alpha_1,\alpha_2,\dots)$ if the constant $C$ depends only on $\alpha_1,\alpha_2,\dots$.\smallskip

 The main result of the paper is the following Runge-type approximation theorem.
 \begin{Th}\label{te1.1}
 Let $K\subset\Di$ and $\widehat U\subset\mathfrak M$ be an open neighbourhood of the closure $\bar K$ of $K$ in $\mathfrak M$.

{\rm (1)} There exist an interpolating sequence $\zeta=\{z_n\}\subset U$ such that $\bar\zeta\subset\widehat U$ and $\bar K\cap\bar\zeta=\emptyset$,
and for each $\varepsilon\in (0,1)$ there exist an interpolating Blashcke product $B_{\zeta_\varepsilon}$ with zero locus $\zeta_\varepsilon$ and a bounded linear operator $L_{\varepsilon}^X: H^\infty(U,X)\to H^\infty(\Di,X)$ depending on  $B_{\zeta_\varepsilon}$, $K$ and $U$ such that the following holds.\smallskip
\begin{itemize}
\item[(i)]
There exists a constant $C=C(K,U)>0$ such that $C\to\infty$ as $\rho(K,\Di\setminus U)\to 0$ and for all $f\in H^\infty(U,X)$,
\begin{equation}\label{eq1.6}
\sup_{z\in K}\left\|f(z)-\frac{(L_{ \varepsilon}^X(f))(z)}{B_{\zeta_\varepsilon}(z)} \right\|_X\le C\varepsilon\|f\|_{H^\infty(U,X)}.
\end{equation}

\item[(ii)]
For some constant $d:=d(K,U)>0$ such that $d\to\infty$ as $\rho(K,\Di\setminus U)\to 0$ and all $f\in H^\infty(U,X)$,
\begin{equation}\label{eq1.7}
\|L_{\varepsilon}^X(f)\|_{H^\infty(\Di,X)}\le |f|_{H^\infty(U,X)|_\Gamma}+C\varepsilon^{d}\|f\|_{H^\infty(U,X)}.
\end{equation}
In particular, 
\[
\|L_{\varepsilon}^X\|\le \delta_{\widehat U,\Gamma}+C\varepsilon^{d},
\]
where $\delta_{\widehat U,\Gamma}$ is   $0$ if $\widehat U\cap \Gamma=\emptyset$ and is $1$ otherwise.\smallskip

\item[(iii)] $\bar\zeta_\varepsilon\cap \bar K=\emptyset$ and for some $r:=r(d)\in (0,1)$ such that $r=O(2^{-d})$ as $d\to\infty$ the pseudohyperbolic disks $D_{r}(z_n)$, $n\in\N$, are mutually disjoint and the boundary circle of each disk $D_{r}(z_n)$ contains $N_\varepsilon:=\lfloor\log_2\frac{1}{\varepsilon}\rfloor +1$ points of $\zeta_\varepsilon$ equally spaced with respect to $\rho$. 

\noindent Moreover, the function $(0,1)\ni\varepsilon\mapsto\delta(\zeta_\varepsilon)\in (0,1]$ is constant on the intervals $[2^{-m},2^{-m+1})$, $m\in\N$, and for each $\varepsilon\in \{2^{-m}\, :\, m\in\N\}$ satisfies
\begin{equation}\label{eq1.8}
(\delta(\zeta))^{a}\frac{1-r^2}{1-r^{2N_\varepsilon }}N_\varepsilon\varepsilon^{\log_2\frac{1}{r(\delta(\zeta))^a}} <\delta(\zeta_\varepsilon)< (\delta(\zeta))^{b}\frac{1-r^2}{1-r^{2N_\varepsilon }}N_\varepsilon\varepsilon^{b\log_2\frac{1}{r(\delta(\zeta))^b }},
\end{equation}
where $a=a(r)\in (1,2 )$, $b=b(r)\in (0,1)$ and $a-1=O\bigl(\frac{1}{\ln\frac{1}{r}}\bigr)$, $1-b=O\bigl(\frac{1}{\ln\frac{1}{r}}\bigr)$ as $r\to 0^+$.

\end{itemize}
(2) Suppose that $\bar K$ is contained in a compact subset of
$\widehat U$ which  is an $n$-ary Cartesian product of holomorphically convex subsets of $\mathfrak M$. Then for each $\varepsilon\in (0,1)$ there exist a bounded linear operator $E_{\varepsilon}^X: H^\infty(U,X)\to H^\infty(\Di,X)$ such that the following holds.
\begin{itemize}
\item[(i)] 
There exists a constant $C_0=C_0(K,U)>0$ such that for all $f\in H^\infty(U,X)$,
\begin{equation}\label{eq1.9}
\sup_{z\in K}\left\|f(z)- (E_{\varepsilon}^X(f))(z)\right\|_X\le C_0\varepsilon\|f\|_{H^\infty(U,X)}.
\end{equation}
 \item[(ii)] For some constants $C_1:=C_1(K,U)>0$, $d_1:=d_1(K,U)>0$
 \[
 \|E_{\varepsilon}^X\|\le C_1\varepsilon^{-d_1}.
 \]
 \end{itemize}
(3) The operators $L_{\varepsilon}^X$ and $E_{\varepsilon}^X$ map
the subspace $H^\infty_{\rm comp}(U,X)$ into $H_{\rm comp}^\infty(\Di,X)$. Hence, analogs of parts (1) and (2) are also valid for $X$-valued $H^\infty_{\rm comp}$ functions. Moreover, these operators are universal in the following sense: if $T:X\to Y$ is a bounded linear operator between complex Banach spaces, then
\[
TL_{\varepsilon}^X=L_{\varepsilon}^Y T\quad {\rm and}\quad 
TE_{\varepsilon}^X=E_{\varepsilon}^YT, 
\]
where $(Tf)(z):=T(f(z))$, $z\in U$, $f\in H^\infty(U,X)$.
\end{Th}
\begin{R}\label{rem1.2}
{\rm A particular case of part (2) of the theorem for $f\in H_{\rm comp}^\infty(U,X)$  was established in \cite[Thm.\,1.7]{Br1}. 
}
\end{R}

\sect{Auxiliary results}
In this section, we collect some results used in the proof of Theorem \ref{te1.1}.
\subsect{Banach-valued $\bar\partial$ equations on $\Di$}
 A subset $S$ of  a metric space $(\mathcal M,d)$
 is said to be $\epsilon$-{\em separated} if $d(x,y)\ge\epsilon$ for all $x, y\in S$, $x\ne y$. A maximal $\epsilon$-separated subset of $\mathcal M$ is said to be an $\epsilon$-{\em chain}. Thus, if $S\subset \mathcal M$ is an $\epsilon$-chain, then $S$ is $\epsilon$-separated and for every $z\in \mathcal M\setminus S$ there is 
 $x\in S$ such that $d(z,x)<\epsilon$.
 Existence of $\epsilon$-chains  follows from the Zorn lemma.\smallskip

 A subset $S\subset\Di$ is said to be {\em quasi-interpolating}, if an
 $\epsilon$-chain of $S$, $\epsilon\in (0,1)$,  with respect to the pseudohyperbolic metric $\rho$, see \eqref{eq1.3}, is an interpolating sequence for $H^\infty$. (In fact, in this case every $\epsilon$-chain of $S$, $\epsilon\in (0,1)$, with respect to $\rho$ is an interpolating sequence for $H^\infty$, this easily follows from \cite[Ch.\,X,\,Cor.\,1.6, Ch.\,VII,\,Lm.\,5.3]{Ga}.)
 
Let $K\subset\Di$ be a Lebesgue measurable subset and $X$ be a complex Banach space. Two $X$-valued functions on $\Di$ are equivalent if they coincide a.e. on $\Di$. The complex Banach space
$L^\infty(K,X)$ consists of equivalence classes of Bochner measurable essentially bounded functions $f:\Di\to X$ equal $0$ a.e.\,\,on $\Di\setminus K$ equipped with norm $\|f\|_{L^\infty(K,X)}:={\rm ess}\sup_{z\in K}\|f(z)\|_X$. Also, we denote
 by $C_\rho(\Di,X)$  the Banach space of bounded continuous functions $f:\Di\rightarrow X$  uniformly continuous with respect to $\rho$ equipped with norm $\|f\|_{C_\rho(\Di,X)}:=\sup_{z\in \Di}\|f(z)\|_X$.

In \cite{Br2} we studied the differential equation
\begin{equation}\label{eq2.2}
\frac{\partial F}{\partial\bar z}=\frac{f(z)}{1-|z|^2},\qquad |z|<1,\quad f\in L^\infty(K,X).
\end{equation}
We proved that if $K$ is quasi-interpolating, then equation \eqref{eq2.2}
has a weak solution $F\in C_\rho(\Di,X)$, i.e., such that
for every $C^\infty$ function $s$ with compact support in $\Di$
\begin{equation}\label{eq2.3}
\iint \limits_{\Di}F(z)\cdot\frac{\partial s(z)}{\partial\bar{z}}\,dz\wedge d\bar{z}=-\iint\limits_{\Di}\frac{f(z)}{1-|z|^2}\,\cdot  s(z)\,dz\wedge d\bar{z},
\end{equation}
given by a bounded linear operator $L_K^X: L^\infty(K,X)\rightarrow  C_\rho(\Di,X)$.  
 Specifically, we obtained the following result.
 \begin{Th}[\mbox{\cite[Thm.\,1.1]{Br2}}]\label{te2.1}
Suppose a quasi-interpolating set $K\subset\Di$ is Lebesgue measurable and $\zeta=\{z_j\}$ is an $\epsilon$-chain of $K$, $\epsilon\in (0,1)$, with respect to $\rho$ such that 
\[
\delta(\zeta):=
\inf_k\prod_{j,\,j\ne k}\rho(z_j,z_k)\ge\delta>0.
\]
There is  a bounded linear operator $L_K^X: L^\infty(K,X)\to C_\rho(\Di,X)$ of norm 
\begin{equation}\label{eq2.4}
\|L_K^X\|\le \frac{c\epsilon}{1-\epsilon}\cdot\max\left\{1,\frac{\log\frac{1}{\delta}}{(1-\epsilon_*)^2}\right\},\quad \epsilon_*:=\max\left\{\frac 1 2,\epsilon\right\},
\end{equation}
for a numerical constant $c<5^2\cdot 10^6$
such that  for every $f\in L^\infty(K,X)$ the function $L_K^X f$ is a weak solution of equation \eqref{eq2.2}. \smallskip
 
The operator $L_K^X$ has the following properties:
\begin{itemize}
\item[(i)]
If $T: X\rightarrow Y$ is a bounded linear operator between complex Banach spaces, then 
\[
TL_K^X=L_K^Y T,
\]
where $(Tf)(z):=T(f(z))$, $z\in\Di$, $f:\Di\to X$;\smallskip
\item[(ii)]
If $f\in  L^\infty(K,X)$ has a compact essential range, then 
the range of $L_K^X f$ is relatively compact;\smallskip
\item[(iii)]
If $f\in L^\infty(X,K)$ is continuously differentiable on an open set $U\subset\Di$, then $L_K^X f$ is continuously differentiable on $U$. 
\end{itemize}
\end{Th}
\subsect{Structure of $\mathfrak M$}
For $m_1,m_2\in \mathfrak M$ the formula
\begin{equation}\label{eq2.5}
\rho(m_1,m_2):=\sup\bigl\{|\hat f(m_2)|\, :\, f\in H^\infty,\, \hat f(m_1)=0,\, \|f\|_{H^\infty}\le 1\bigr\}
\end{equation}
gives an extension of the pseudohyperbolic metric $\rho$  to $\mathfrak M\times\mathfrak M$. The extended function is lower semicontinuous on $\mathfrak M\times\mathfrak M$ (see \cite[Thm.\,6.2]{H}) and determines a metric on $\mathfrak M$ with the property that  any two open balls of radius $1$ are either equal or disjoint.  The {\em Gleason part} of $m\in \mathfrak M$ is then defined by $P(m):=\{m'\in \mathfrak M\, :\, \rho(m',m)<1\}$. Hoffman's classification of Gleason parts \cite{H} shows that there are only two cases: either $P(m)=\{m\}$ or $P(m)$ is an analytic disk. The former case means that there is a {\em parameterization} of $P(m)$, i.e., a continuous one-to-one and onto map $L:\Di\to P(m)$ such that  $\hat f\circ L\in H^\infty$ for every $f\in H^\infty$. 
By $\mathfrak M_a$  we denote the union of all non-trivial (analytic disks) Gleason parts of $\mathfrak M$. It is known that $\mathfrak M_a\subset \mathfrak M$ is open and $P(m)\subset \mathfrak M_a$ if and only if $m$ belongs to the closure of an interpolating sequence for $H^\infty$. 

According to Hoffman, the base of topology of $\mathfrak M_a$ consists of 
sets of the form $\{x \in\mathfrak M_a\, :\, |\hat b(x)| < \varepsilon\}$, where $b$ is an interpolating Blaschke product. This follows from the fact that for all sufficiently small $\varepsilon$ the set $b^{-1}(\Di_\varepsilon) \subset\Di$ is biholomorphic to $\Di_\varepsilon \times b^{-1}(0)$, see, e.g., \cite[Ch.\,X, Lm.\,1.4]{Ga}. Hence, $\{x \in\mathfrak M_a\, :\, |\hat b(x)| < \varepsilon\}$ is homeomorphic to $\Di_\varepsilon \times \hat b^{-1}(0)$. In turn, the base of topology of $\mathfrak M$ consists of 
sets of the form $\{x \in\mathfrak M\, :\, |\hat f(x)| < \varepsilon\}$, $f\in H^\infty$.
This follows from Su\'{a}rez's theorem \cite[Thm.\,2.4]{S0} which states that the algebra $H^\infty$ is {\em separating}, i.e., that if $K$ is a closed subset of $\mathfrak M$ and $x\in\mathfrak M\setminus K$, then there exists a function $g\in H^\infty$ such that $\hat g(x)\notin \hat g(K)$. An alternative proof of this fact can be easily obtained from Su\'arez's Runge-type approximation theorem \cite[Thm.\,3.3(I)]{S1} presented in section~1.1. Indeed, consider nonintersecting open neighbourhoods $O_1$ and $O_2$ of $K$ and $x$, respectively. Let $f$ be a continuous function on $O_1\sqcup O_2$ that is equal to $1$ on $O_1$ and equal to $0$ on $O_2$. Then  \cite[Thm.\,3.3(I)]{S1} implies that there exist $h\in H^\infty$ and an interpolating Blaschke product $b$ such that $\inf_{K\cup\{x\}}|\hat b|=:\delta>0$ and
\[
\sup_{K\cup\{x\}} |f-\hat h/\hat b|<\frac 14.
\]
This implies for $g:=h-\frac{\hat h(x)}{\hat b(x)}b$,
\[
\inf_K|\hat g|\geq \inf_K\left|\frac{\hat h}{\hat b}-\frac{\hat h(x)}{\hat b(x)}\right|\cdot\inf_K|\hat b|\ge \left(\inf_K\left|\frac{\hat h}{\hat b}\right|-\left|\frac{\hat h(x)}{\hat b(x)}\right|\right)\delta\ge \frac{\delta}{2},
\]
so that $0=\hat g(x)\notin \hat g(K)\subset\Co\setminus\Di_{\frac{\delta}{2}}$, as required.\smallskip

By $\mathfrak M_s:=\mathfrak M\setminus\mathfrak M_a$ we denote the set of trivial (one-point) Gleason parts. By virtue of Su\'{a}rez's result \cite[Thm.\,3.4]{S2}, the set $\mathfrak M_s$ is totally disconnected.
Using this fact, the following result was obtained in \cite{Br1} (for the proof see \cite[Prop.\,3.2,\,Lm.\,4.1]{Br1}).

\begin{Lm}\label{le2.2}
For a finite open cover of $\mathfrak M$, there exist an open finite refinement $(\widehat V_i)_{i\in I}$  and a smooth partition of unity $\{\rho_i\}_{i\in I}$ subordinate to the cover $(V_i)_{i\in I}$ of $\Di$, where $V_i:=\widehat V_i\cap\Di$, $i\in I$,
such that 
\begin{itemize}
\item[(a)]
\hspace*{1mm} All nonempty  $\widehat V_i\cap \widehat V_j$, $i\ne j$,
are relatively compact subsets of $\mathfrak M_a$;\smallskip
\item[(b)] \hspace*{1mm}
The closure of the support of $\rho_i$ in $\mathfrak M$ is contained in $\widehat V_i$, and 
\begin{equation}\label{eq2.6}
\frac{\partial\rho_i}{\partial\bar z}=\frac{g_{i}(z)}{1-|z|^2},\quad
z\in \mathbb{D},
\end{equation}
where $g_i \in C_\rho(\mathbb{D})\cap C^\infty(\Di)$, $i\in I$.
\end{itemize}
\end{Lm}
In addition, we have
\begin{C}\label{cor2.3}
Every function $\rho_i$ in the lemma extends to a continuous function $\hat\rho_i$ on $\mathfrak M$ with support in $\widehat V_i$. 
\end{C}
\noindent Hence,  $\{\hat \rho_i\}_{i\in I}$ 
is a continuous partition of unity subordinate to the cover $(\widehat V_i)_{i\in I}$ of $\mathfrak M$.
\begin{proof}
According to (a), the closure of the support of $\rho_i$, $\overline{{\rm supp}(\rho_i)}$, is a subset of the set $\cup_{j\in I} \widehat V_i\cap \widehat V_j$, which according to (b) is a subset of $\mathfrak M_a$. Thus, by Proposition \ref{prop2.4} proved in section~2.3 below, ${\rm supp}(\rho_i)$ is a quasi-interpolating set. In particular, by virtue of \cite[Cor.\,1.5]{Br2} (a corollary of Theorem \ref{te2.1}), equation \eqref{eq2.6} has a smooth solution $F_i\in C_\rho(\Di)\, (:=C_\rho(\Di,\Co))$ which extends to a continuous function $\hat F_i\in C(\mathfrak M)$. Since $\rho_i$ is also a bounded solution of \eqref{eq2.6}, $h_i:=\rho_i-F_i\in H^\infty$. Thus, $h_i$ has an extension $\hat h_i\in C(\mathfrak M)$ and therefore $\hat\rho_i:=\hat F_i+\hat h_i\in C(\mathfrak M)$ is a continuous extension of $\rho_i$.
\end{proof}

\subsect{Characterization of quasi-interpolating sets}
We require the following result.
\begin{Prop}\label{prop2.4}
A subset of $\Di$ is quasi-interpolating if and only if its closure in $\mathfrak M$ belongs to $\mathfrak M_a$. 
\end{Prop}
\begin{proof}
Let $K\subset\Di$ be such that $\bar K\subset\mathfrak M_a$. Fix $\epsilon>0$ and let $\zeta=\{z_n\}\subset K$ be an $\epsilon$-chain with respect to $\rho$. We have to prove that $\zeta$ is an interpolating sequence for $H^\infty$. 

Let $x\in\mathfrak M_a$ be a limit point of $\zeta$. By Hoffman's theorem, see, e.g., \cite[Ch.\,X]{Ga}, there exists an interpolating sequence $\zeta_x\subset\Di$ such that $x$ is a limit point of $\zeta_x$. Let $B_{\zeta_x}$ be the interpolating Blaschke product with zeros at $\zeta_x$. Then $\hat{B}_{\zeta_x}(x)=0$. Next, according to \cite[Ch.\,X,\,Lm.\,1.4,\,Ch.\,VII,\,Thm.\,5.1]{Ga} (see Lemma \ref{lem2.5} in section~2.4 below), for a sufficiently small $c>0$ the set $\widehat W_{x}(c):=\{m\in\mathfrak M\, :\, |\hat B_{\zeta_x}(m)|<c\}$ is an open neighbourhood of $x$ in $\mathfrak M_a$ such that $W_x:=\widehat W_x\cap\Di$ is the disjoint union of open subsets $V_n$, $n=1,\dots$, and $b_x$ maps each $V_n$ biholomorphically onto the open disk $\Di_c$. Moreover, the pseudohyperbolic diameter $d_\rho(V_n)$ of each $V_n$ is less than $\epsilon$ and every sequence $\{w_n\}\subset\Di$ such that $w_n\in V_n$, $n=1,\dots$, is interpolating for $H^\infty$. Since $x$ is also a limit point of $\zeta$, $W_x\cap\zeta\ne\emptyset$. Since $d_\rho(V_n)<\epsilon$ and $\rho(z_i,z_j)\ge\epsilon$ for all $i\ne j$ (because $\zeta$ is an $\epsilon$-chain), each $V_n$ contains at most one point of $\zeta$. Thus, $W_x\cap\zeta$ is an interpolating sequence for $H^\infty$. Using compactness of $\bar\zeta$ we find finitely many points $x_1,\dots, x_k\in\bar\zeta$ such that the corresponding sets $\zeta_i:=W_{x_i}\cap\zeta$, $1\le i\le k$, cover $\zeta$. Since each $\zeta_i$ is an interpolating sequence for $H^\infty$  and $\rho(z_i,z_j)\ge\epsilon$ for all $i\ne j$, the sequence $\zeta$ is also interpolating for $H^\infty$, see, e.g.,
\cite[Ch.\,VII,\,Problem 2]{Ga}.

Conversely, suppose that the set $K\subset\Di$ is quasi-interpolating. Let $\epsilon\in (0,1)$ and let $\zeta$ be an $\epsilon$-chain of $K$ with respect to $\rho$ which is an interpolating sequence for $H^\infty$.
Let $x\in\bar K$ and 
$(x_\alpha)\subset K$ be a net converging to $x$ in $\mathfrak M$. Then for each $x_\alpha$ there exists $z_\alpha\in\zeta$ such that $\rho(x_\alpha,z_\alpha)<\epsilon$. Passing to a subnet of $(z_\alpha)$, if necessary, without loss of generality we can assume that $(z_\alpha)$ converges to some $m\in\mathfrak M_a$. Then by the lower semicontinuity of $\rho$ on $\mathfrak M\times\mathfrak M$, see \eqref{eq2.5}, we obtain 
\[
\rho(x,m)\le \liminf_\alpha\rho(x_\alpha,z_\alpha)\le\epsilon<1.
\]
Hence $y\in P(m)\subset\mathfrak M_a$. This shows that $\bar K\subset\mathfrak M_a$ and completes the proof of the proposition.
\end{proof}
\subsect{}
We also use the following results.
 \begin{Lm}[\mbox{\cite[Ch.\,X,\,Lm.\,1.4,\, Ch.\,VII,\,Lm.\,5.3]{Ga}}]\label{lem2.5}
Let $B_\zeta$ be the interpolating Blaschke product with zeros $\zeta=\{z_n\}$  such that
\[
\delta(\zeta)=\inf_n\, (1-|z_n|^2)|B_\zeta'(z_n)|\ge\delta>0.
\]
Suppose $\lambda\in (0,1)$ and $r:=r(\lambda)\in (0,1)$ satisfy
\[
\frac{2\lambda}{1+\lambda^2}<\delta\quad{\rm and}\quad r=\frac{\delta-\lambda}{1-\lambda\delta}\lambda.
\]
Then 
\begin{itemize}
\item[(i)]
$B_\zeta^{-1}(\Di_r)=\{z\in\Co\, :\, |B_\zeta(z)|<r\}$ is the union of pairwise disjoint domains $V_{\zeta,n}\,(\ni z_n)$ such that 
\[
V_{\zeta,n}\subset D(z_n,\lambda);\smallskip
\]
\item[(ii)]
$B_\zeta$ maps every $V_{\zeta, n}$ biholomorpically onto $\Di_r$;\smallskip
\item[(iii)] 
Every sequence $\omega=\{ w_n\}$ with $w_n\in D(z_n,\lambda)$ for all $n$,
is  interpolating for $H^\infty$ and
 \[
\delta(\omega) \ge\frac{\delta-\frac{2\lambda}{1+\lambda^2}}{1-\frac{2\delta \lambda}{1+\lambda^2}}.
\]
\end{itemize}
\end{Lm}
By $b_{\zeta, n} :\Di_r\to V_{\zeta, n}$ we denote the (holomorphic) inverse of $B_\zeta|_{V_n}$.
Since $B_\zeta(z_n)=0$ and $\|B_\zeta\|_{H^\infty}=1$,  by the Schwarz-Pick theorem $D(z_n,r)\subset b_{\zeta, n}(\Di_r)\subset V_{\zeta, n}\, (\subset D(z_n,\lambda))$.
\begin{Lm}[\mbox{\cite[Ch.\,X,\,Cor.\,1.6]{Ga}}]\label{lem2.6}
Every interpolating sequence $\zeta$ containing at least two elements can be represented as the disjoint union of interpolating sequences $\zeta_1$ and $\zeta_2$ such that
\[
\delta(\zeta_j)\ge\sqrt{\delta(\zeta)},\quad j=1,2.
\]
\end{Lm}

\sect{Proof of Theorem \ref{te1.1}(1)}
\subsect{Construction of linear approximating operators} We set $\widehat U':=\mathfrak M\setminus \bar K$.
Applying Lemma \ref{le2.2} to the open cover
$(\widehat U,\widehat U')$ of $\mathfrak M$,
 we find an open finite refinement  $(\widehat V_i)_{i\in I}$ of the cover and a smooth partition of unity $\{\rho_i\}_{i\in I}$ subordinate to the open cover $(V_i)_{i\in I}$ of $\Di$, where $V_i:=\widehat V_i\cap\Di$, $i\in I$, satisfying conditions (a) and (b) of the lemma. Let $I_1\subset I$ be the subset of indices such that $V_i\subset U$ for all $i\in I_1$ and $I_2:=I\setminus I_1$. We set
\[
\widehat V:=\bigcup_{i\in I_1}\widehat V_i,\qquad \widehat V':=\bigcup_{i\in I_2} \widehat V_i,\qquad V=\widehat V\cap\Di,\qquad V'=\widehat V'\cap\Di,
\] 
and 
\[
\rho_{ V}:=\sum_{i\in I_1}\rho_i,\qquad \rho_{ V'}:=\sum_{i\in I_2}\rho_i.
\]
Then $(\widehat V,\widehat V')$ is an open cover of $\mathfrak M$ such that $\bar K\subset\widehat V$, and $\{\rho_{ V},\rho_{ V'}\}$ is a partition of unity subordinate to the open cover 
$(V, V')$ of $\Di$ satisfying also conditions (a) and (b), i.e.,
\begin{equation}\label{e3.1}
\widehat V\cap \widehat V'\Subset\mathfrak M_a\quad \textrm {and }\quad
\frac{\partial\rho_{ V}}{\partial\bar z}=-\frac{\partial \rho_{ V'}}{\partial\bar z}=\frac{g(z)}{1-|z|^2},\ \, z\in\Di,  \textrm{ \ where\ } g\in C_\rho^\infty(\Di).
\end{equation}
Moreover, the closure of the support of $g$ in $\mathfrak M$, $\overline{{\rm supp}(g)}$,  is a compact subset of $\widehat V\cap \widehat V'$.  Thus, $\overline{{\rm supp}(g)}\cap\bar K=\emptyset$.
In turn, since $\overline{{\rm supp}(g)}$ is a compact subset of $\mathfrak M_a$, by Proposition \ref{prop2.4} the set ${\rm supp}(g)$ is quasi-interpolating. 

Let
\begin{equation}\label{e3.2}
\mathfrak d:=\rho(K,{\rm supp}(g))
\end{equation}
be the pseudohyperbolic distance between $K$ and ${\rm supp}(g)$. Since 
$\overline{{\rm supp}(g)}\cap\bar K=\emptyset$ and $\rho$ is lower semicontinuous on $\mathfrak M\times\mathfrak M$, 
\[
0<\mathfrak d<1.
\]
We set
\begin{equation}\label{e3.3}
 \mathfrak e:=\frac{\mathfrak d}{4}.
\end{equation}

Let $\zeta=\zeta(\mathfrak e)$ be an $\mathfrak e$-chain of ${\rm supp}(g)$ with respect to $\rho$. By the argument of the proof of Proposition \ref{prop2.4}, $\zeta $ is an interpolating sequence for $H^\infty$ with the  characteristic $\delta(\zeta  )\in (0,1]$. 

Let
\begin{equation}\label{e3.3a}
c:=\frac{\frac{99}{100}+\frac{2\frac{1}{16}}{1+ (\frac{1}{16} )^2}}{1+\frac{2\frac{99}{100} \frac{1}{16}}{1+(\frac{1}{16})^2}}=0.99220590273\dots .
\end{equation}

Using Lemma \ref{lem2.6} repeatedly, we decompose $\zeta  $ into pairwise disjoint interpolating subsequences $\zeta_i=\{z_{n,i}\}_{n\in\N}$, $1\le i\le k$, where $k\le 2(\lfloor\log_c\delta(\zeta)\rfloor+1)$, whose characteristics $\delta(\zeta_i)$
satisfy
\begin{equation}\label{e3.3b}
\delta(\zeta_i)\ge c.
\end{equation}
Then we have
\begin{Lm}\label{lem3.1}
\[
\frac{\delta(\zeta_i)-\frac{2\frac{\mathfrak e}{4}}{1+ (\frac{\mathfrak e}{4} )^2}}{1-\frac{2\delta(\zeta_i) \frac{\mathfrak e}{4}}{1+(\frac{\mathfrak e}{4})^2}}\ge\frac{99}{100} .
\]
\end{Lm}
\begin{proof}
By definition, $\frac{\mathfrak e}{4}\le\frac{1}{16}$, see \eqref{e3.2}, \eqref{e3.3},  the function $h(x)=\frac{2x}{1+x^2}$, $x\in [0,1)$, is increasing, and the function $g(x,y)=\frac{y-x}{1-yx}$, $x\in [0,1)$, $y\in [0,1)$, is decreasing in $x$ and increasing in $y$. Hence, due to \eqref{e3.3a},
\[
\frac{\delta(\zeta_i)-\frac{2\frac{\mathfrak e}{4}}{1+ (\frac{\mathfrak e}{4} )^2}}{1-\frac{2\delta(\zeta_i) \frac{\mathfrak e}{4}}{1+(\frac{\mathfrak e}{4})^2}}=g(h(\mbox{$\frac{\mathfrak e}{4}$}),\delta(\zeta_i))\ge g(h(\mbox{$\frac{1}{16}$}),\delta(\zeta_i))\ge g(h(\mbox{$\frac{1}{16}$}),c)=\frac{c-\frac{2\frac{1}{16}}{1+ (\frac{1}{16} )^2}}{1-\frac{2c \frac{1}{16}}{1+(\frac{1}{16})^2}}=0.99,
\]
as required.
\end{proof}

Let $\varepsilon\in (0,1)$, $
N_\varepsilon:=\left\lfloor\log_2\frac{1}{\varepsilon}\right\rfloor +1$.
We consider the sequences $\zeta_{i,j,\varepsilon}:=\{z_{n,i, j,\varepsilon}\}_{n\in\N}$, $1\le j\le N_\varepsilon$, $1\le i\le k$, and $\zeta_{i,\varepsilon}=\bigsqcup_{j=1}^{N_\varepsilon} \zeta_{i,j,\varepsilon}$, $1\le i\le k$, where
\begin{equation}\label{e3.5}
z_{n,i,j,\varepsilon}:=\frac{\frac{\mathfrak e}{4}\omega_{n,i,j}+z_{n,i}}{1+\bar z_{n,i}\cdot \frac{\mathfrak e}{4}\omega_{n,i,j}},
\end{equation}
and  $\{\omega_{n, i,j}\}_{j=1}^{N_\varepsilon}$, $n\in\N$, $1\le i\le k$,  are arbitrary sequences of equally spaced points on the unit circle, counted counterclockwise.

In particular, for all $i$ and $\varepsilon$ the points $z_{n,i,1,\varepsilon},\dots, z_{n, i, N_\varepsilon, \varepsilon}$ lie on the boundary circle of the disk $D(z_{n,i},\frac{\mathfrak e}{4})$ and are equally spaced with respect to the metric $\rho$. 

We apply Lemma \ref{lem2.5}(iii) with $\delta:=\delta(\zeta_i)$ and $\lambda= \frac{\mathfrak e}{4}$. Since by \eqref{e3.3a}
\[
\frac{2\lambda}{1+\lambda^2}=\frac{2\frac{\mathfrak e}{4}}{1+ (\frac{\mathfrak e}{4} )^2}\le \frac{2\frac{1}{16}}{1+ (\frac{1}{16} )^2}<\delta(\zeta_i),
\]
the assumption of the lemma is true. Hence, Lemma \ref{lem2.5}(iii) and 
Lemma \ref{lem3.1} imply that each sequence $\zeta_{i,j,\varepsilon}$ is interpolating and
\begin{equation}\label{e3.6}
\delta(\zeta_{i,j,\varepsilon})\ge\frac{99}{100}.
\end{equation}

Let
\begin{equation}\label{e3.7}
W_i:=\bigsqcup_{n}D(z_{n,i},\mathfrak e)
\end{equation}
be the pseudohyperbolic $\mathfrak e$-neighbourhood of $\zeta_i$. By  definition, $(W_i)_{i=1}^k$ is an open cover of ${\rm supp}(g)$ and each $\zeta_i$ is an $\mathfrak e$-chain of $W_i$, $1\le i\le k$.  
\begin{Lm}\label{lem3.2}
Let $B_{\zeta_{i,j,\varepsilon}}$ be the interpolating Blaschke product having simple zeros at points of $\zeta_{i,j,\varepsilon}$. Then
\begin{equation}\label{e3.8}
\sup_{W_i}|B_{\zeta_{i,j,\varepsilon}}|\le \frac{5\mathfrak e}{4}\quad {\rm and}\quad  \inf_{K}|B_{\zeta_{i,j,\varepsilon}}|\ge \frac{5\mathfrak e}{2}.
\end{equation}
\end{Lm}
\begin{proof}
We apply Lemma \ref{lem2.5} to $B_{\zeta_{i,j,\varepsilon}}$ with $\delta:=\frac{99}{100}$ and $\lambda=\frac{8}{10}$, see \eqref{e3.6}. Then the corresponding parameter $r$ satisfies
\[
r=\frac{19}{100-\frac{99\cdot 4}{5}}\cdot\frac{4}{5}>\frac{5}{8}.
\]
Since in this case
\[
\frac{2\lambda}{1+\lambda^2}=\frac{\frac 85}{1+\frac{16}{25}}=\frac{40}{41}<\frac{99}{100}=:\delta,
\]
the assumption of the lemma is true. Hence, due to part (i) of the lemma  we obtain (in the notation of the lemma) for all $n$,
\begin{equation}\label{e3.9}
b_{\zeta_{i,j,\varepsilon}, n}(\Di_{\frac{5 }{8}})\subset V_{\zeta_{i,j,\varepsilon}, n} \subset D(z_{n,i,j,\varepsilon},\mbox{$\frac{4}{5}$}).
\end{equation}
Since $\frac{5\mathfrak e}{2}\le\frac 58$, applying the Schwarz-Pick theorem
to the map $b_{\zeta_{i,j,\varepsilon}, n}:\Di_{\frac{5 }{8}}\to D(z_{n,i,j,\varepsilon},\mbox{$\frac{4}{5}$})$ and  using that
\[
\rho(z_{n,i},w)\le \rho(z_{n,i,j,\varepsilon},z_{n,i})+\rho(z_{n,i,j,\varepsilon},w)=\frac{\mathfrak e}{4}+\rho(z_{n,i,j,\varepsilon},w),
\]
and the definitions of $\mathfrak d$ and $\mathfrak e$, see \eqref{e3.2}, \eqref{e3.3},  we get from   \eqref{e3.9} for all $n$,
\[
b_{\zeta_{i,j,\varepsilon}, n}(\Di_{\frac{5\mathfrak e}{2}})\subset D(z_{n,i,j,\varepsilon},\mbox{$\frac{5\mathfrak e}{2}\cdot\frac{8}{5}\cdot\frac{4}{5}$})\subset D(z_{n,i},\mbox{$\frac{16\mathfrak e}{5}+\frac{\mathfrak e}{4}$})\subset D(z_{n,i},\mbox{$\frac{69\mathfrak d}{80}$})\subset\Di\setminus K.
\]
From here and from the implication $D(z_{n,i,j,\varepsilon}, \frac{5\mathfrak e}{4} )\subset b_{\zeta_{i,j,\varepsilon}, n}(\Di_{\frac{5\mathfrak e}{4}})$ we obtain for all $i,j,\varepsilon$,
\begin{equation}\label{e3.10}
\sup_{D(z_{n,i,j,\varepsilon}, \frac{5\mathfrak e}{4} )}|B_{\zeta_{i,j,\varepsilon}}|\le \frac{5\mathfrak e}{4}\quad {\rm and}\quad  \inf_{K}|B_{\zeta_{i,j,\varepsilon}}|\ge \frac{5\mathfrak e}{2}.
\end{equation}
Finally, by the triangle inequality for $\rho$,  
\[
W_i=\bigsqcup _n D(z_{n,i},\mathfrak e)\subset \bigsqcup _n D(z_{n,i,j,\varepsilon}, \mbox{$\frac{5\mathfrak e}{4}$} ).
\]
This and \eqref{e3.10} imply inequality \eqref{e3.8}.
\end{proof}

We set 
\begin{equation}\label{e3.11}
W_i ':=W_i\setminus\cup_{j=1}^{i-1}W_j,
\end{equation}
and denote by $\chi_i$ the characteristic function of $W_i'$, $1\le i\le k$. Clearly, each set $W_i'$ is Lebesgue measurable and therefore each $\chi_i\in L^\infty(W_i')$.

Let $L_{W_i}^X: L^\infty(W_i,X)\to C_\rho(\Di,X)$, $1\le i\le k$, be the bounded linear operators of Theorem \ref{te2.1}. Since each $\zeta_i$ is an $\mathfrak e$-chain of $W_i$, where $\mathfrak e\le\frac{1}{4}$, see \eqref{e3.3}, and $\delta(\zeta_i)\ge c>0.99$, see \eqref{e3.3a}, the theorem implies that
\begin{equation}\label{e3.12}
\|L_{W_i}^X\|\le A:=\frac{5^2\cdot 10^6}{3},\quad 1\le i\le k.
\end{equation}

Let  
\begin{equation}\label{e3.13}
B_{\zeta_{i,\varepsilon}}=\prod_{j=1}^{N_\varepsilon}B_{\zeta_{i,j,\varepsilon}}\quad {\rm and}\quad B_{\zeta_{\varepsilon}}:=\prod_{i=1}^k B_{\zeta_i,\varepsilon}
\end{equation}
be the Blaschke products having simple zeros at points of $\zeta_{i,\varepsilon}=\bigsqcup_{j=1}^{N_\varepsilon}\zeta_{i,j,\varepsilon}$, $1\le i\le k$, and $\zeta_{\varepsilon}:=\bigsqcup_{i=1}^{k}\zeta_{i,\varepsilon}$, respectively. Then the  required family of approximating operators of Theorem \ref{te1.1}\,(1),  $L_{\varepsilon}^X: H^\infty(U,X)\to C_\rho(\Di,X)$, is given by the formula:
\begin{equation}\label{e3.14}
L_{\varepsilon}^X(f):=B_{{\zeta_\varepsilon}}\cdot\left(\rho_V   f-\sum_{j=1}^k
 \frac{ L_{W_j}^X\bigl(B_{\zeta_{j,\varepsilon}}\, \chi_j  g  f\bigr)}{B_{\zeta_{j,\varepsilon}}}\right),\quad f\in H^\infty(U,X).
\end{equation}
\subsect{Proofs of parts (i) and (ii) of Theorem \ref{te1.1}(1)}
The required results follow from
\begin{Prop}\label{prop3.3}
\begin{itemize}
\item[(a)]
Every  $L_{\varepsilon}^X$ maps $H^\infty(U,X)$ to $H^\infty(\Di,X)$
and satisfies for all $f\in H^\infty(U,X)$,
\begin{equation}\label{e3.15}
\|L_{\varepsilon}^X(f)\|_{H^\infty(\Di,X)}\le |f|_{H^\infty(U,X)|_\Gamma}+C\varepsilon^{d}\|f\|_{H^\infty(U,X)},
\end{equation}
where $C:=kA\|g\|_{C_\rho(\Di)}$ and $d:=\log_2\frac{4}{5\mathfrak e}$.

In particular, 
\[
\|L_\varepsilon^X\|\le \delta_{\widehat U,\Gamma}+C\varepsilon^{d}.
\]
\item[(b)]
\begin{equation}\label{e3.16}
\sup_{z\in K}\left\|f(z)-\frac{(L_\varepsilon^X(f))(z)}{B_{{\zeta_\varepsilon}}(z)} \right\|_X\le C\varepsilon \|f\|_{H^\infty(U,X)}.
\end{equation}
\end{itemize}
\end{Prop}
\begin{proof}
(a)   Since $B_{\zeta_{j,\varepsilon}}\chi_i  g   f\in L^\infty(W_i,X)$  for all $f\in H^\infty(U,X)$, the operator $L_\varepsilon^X$ is well-defined.
Then by the definition of operators $L_{W_j}^X$, see Theorem \ref{te2.1}, we have (in the weak sense) for all $f\in H^\infty(U,X)$, 
\[
\begin{split}
 \frac{\partial\bigl(L_\varepsilon^X(f)\bigr)}{\partial\bar z}=
&
\frac{\partial\bigl(\rho_V  B_{\zeta_{\varepsilon}}  f  \bigr)}{\partial\bar z}
- \sum_{j=1}^k\left(\prod_{i,\,i\ne j}B_{\zeta_{i,\varepsilon}} \right)\cdot \frac{B_{\zeta_{j,\varepsilon}}\chi_j \, g  f}{1-|z|^2}\medskip\\
=&
\frac{\partial\bigl(\rho_V  B_{\zeta_{\varepsilon}}  f  \bigr)}{\partial\bar z}
- \frac{B_{\zeta_{\varepsilon}}\,  g  f}{1-|z|^2}=\frac{\partial\bigl(\rho_V  B_{\zeta_{\varepsilon}}  f  \bigr)}{\partial\bar z}
- \frac{\partial\bigl(\rho_V  B_{\zeta_{\varepsilon}}  f  \bigr)}{\partial\bar z}=0.
\end{split}
\]
Here we used  that $\sum_{i=1}^k\chi_i g=g$ because ${\rm supp}(g)\subset \cup_{i=1}^k W_i=\cup_{i=1}^k W_i'$, see \eqref{e3.2}, \eqref{e3.3}, \eqref{e3.7}.

This shows that every  $L_\varepsilon^X$ maps $H^\infty(U,X)$ to $H^\infty(\Di,X)$.

Let $\varphi\in X^*$. According to Theorem \ref{te2.1}(i)  and \cite[Cor.\,1.5]{Br1} for each $1\le j\le k$, 
\[
\varphi\bigl( L_{W_j}^X\bigl(B_{\zeta_{j,\varepsilon}}\, \chi_i  g  f \bigr)\bigr)=L_{W_j}^\Co\bigl(\varphi\bigl(B_{\zeta_{j,\varepsilon}}\, \chi_i  g  f\bigr) \bigr)\in C_\rho(\Di)
\]
and this function extends to a continuous function on $\mathfrak M$.  Similarly, by Corollary \ref{cor2.3}, $\rho_V$ has a continuous extension $\hat\rho_V$ to $\mathfrak M$. 
These and \eqref{e3.14} imply that $L_{W_j}^X\bigl(B_{\zeta_{j,\varepsilon}}\, \chi_i  g  f\bigr)$ and $L_\varepsilon^X(f)$ extend to   weak-$*$ continuous functions on $\mathfrak M$ with values in the bidual $X^{**}$ of $X$ denoted by $\hat{L}_{W_j}^X\bigl(B_{\zeta_{j,\varepsilon}}\chi_i  g  f\bigr)$ and $\hat{L}_{\varepsilon}^X(f) $, respectively, such that
\begin{equation}\label{e3.17}
\begin{split}
\sup_{x\in\mathfrak M}\bigl\|\bigl(\hat{L}_{W_j}^X\bigl(B_{\zeta_{j,\varepsilon}}  \chi_i  g  f\bigr)\bigr)(x)\bigr\|_{X^{**}}=&
\|L_{W_j}^X\bigl(B_{\zeta_{j,\varepsilon}}  \chi_i  g  f\bigr)\|_{C_\rho(\Di,X)}\medskip\\
\sup_{x\in \Gamma}\,\bigl\|\bigl(\hat{L}_{\varepsilon}^X(f)\bigr)(x)\bigr\|_{X^{**}}=&
\| L_\varepsilon^X(f)\|_{H^\infty(\Di,X)}.
\end{split}
\end{equation}
Here we used that  $\Gamma$ is the Shilov boundary of $H^\infty$.

Now, we obtain by \eqref{e3.17}, \eqref{e3.13}, \eqref{e3.14} and Lemma \ref{lem3.2} for each $f\in H^\infty(U,X)$,
\[
\begin{array}{l}
\displaystyle
\|L_\varepsilon^X(f)\|_{H^\infty(\Di,X)}=\sup_{x\in\Gamma}\bigl\|\bigl(\hat{L}_{\varepsilon}^X(f)\bigr)(x)\bigr\|_{X^{**}}\medskip\\
\displaystyle \le
\sup_{x\in\Gamma} \left(\bigl\|\bigl(\hat\rho_V  \hat B_{\zeta_{ \varepsilon}} \hat f\bigr)(x)\bigr\|_{X^{**}} 
+ \sum_{j=1}^k
 \frac{|\hat B_{\zeta_{\varepsilon}}(x)| \cdot \bigl\| \bigl(\hat{L}_{W_j}^X\bigl(B_{\zeta_{j,\varepsilon}}  \chi_j  g  f \bigr)\bigr)(x)\bigr\|_{X^{**}}}{|\hat B_{\zeta_{j,\varepsilon}}(x)|}\right)\medskip\\
 \displaystyle= \sup_{x\in\Gamma}\left(\left\|\bigl(\hat\rho_V   \hat f\bigr)(x)  
 \right\|_{X^{**}}+\sum_{j=1}^k
 \bigl\| \bigl(\hat{L}_{W_j}^X\bigl(B_{\zeta_{j,\varepsilon}}  \chi_j  g  f \bigr)\bigr)(x)\bigr\|_{X^{**}}\right)\le |f|_{H^\infty(U,X)|_\Gamma}\medskip\\
 \displaystyle +\sum_{j=1}^k
 \sup_{x\in\mathfrak M}\bigl\| \bigl(\hat{L}_{W_j}^X\bigl(B_{\zeta_{j,\varepsilon}}  \chi_j  g  f \bigr)\bigr)(x)\bigr\|_{X^{**}}= |f|_{H^\infty(U,X)|_\Gamma}+\sum_{j=1}^k
\|L_{W_j}^X\bigl(B_{\zeta_{j,\varepsilon}} \chi_i  g  f\bigr)\|_{C_\rho(\Di,X)}\medskip\\
 \displaystyle \le |f|_{H^\infty(U,X)|_\Gamma} + A\sum_{j=1}^k \|B_{\zeta_{j,\varepsilon}} \chi_j  g  f\|_{C_\rho(\Di,X)}\le |f|_{H^\infty(U,X)|_\Gamma} \medskip\\
 \displaystyle  
 +kA\|g\|_{C_\rho(\Di)}\cdot\left(\frac{5\mathfrak e}{4}\right)^{N_\varepsilon}\cdot\|f\|_{H^\infty(U,X)} \le  |f|_{H^\infty(U,X)|_\Gamma}+C\varepsilon^{d}\|f\|_{H^\infty(U,X)},
\end{array}
\]
where $C:=kA\|g\|_{C_\rho(\Di)}$ and $d:=\log_2\frac{4}{5\mathfrak e}$.

This proves inequality \eqref{e3.15}.

In turn, using \eqref{e3.13}, Lemma \ref{lem3.2} and the fact that $K\subset V$, we obtain 
\[
\begin{array}{l}
\displaystyle \sup_{z\in K}\left\|f(z)-\frac{(L_\varepsilon^X(f))(z)}{B_{\zeta_{ \varepsilon}}(z)} \right\|_X=\sup_{z\in K}\left\|\sum_{j=1}^k
 \frac{   \bigl(L_{W_j}^X\bigl(B_{\zeta_{j,\varepsilon}}  \chi_j  g  f \bigr)\bigr)(z)}{  B_{\zeta_{j,\varepsilon}}(z)}\right\|_{X}\medskip\\
 \displaystyle \le kA\|g\|_{C_\rho(\Di)}\left(\frac{5\mathfrak e}{4}\right)^{N_\varepsilon}\cdot\left(\frac{2}{5\mathfrak e}\right)^{N_\varepsilon}\|f\|_{H^\infty(U,X)} \le C\varepsilon\|f\|_{H^\infty(U,X)},
 \end{array}
\]

This completes the proof of \eqref{e3.16} and  the proposition.
\end{proof}

Thus, parts (i) and (ii) of Theorem \ref{te1.1}(1) are proved.
\subsect{Proof of part (iii) of Theorem \ref{te1.1}(1)}
In our notation, $\zeta=\{z_n\}_{n\in\N}$,
$r:=5\cdot 2^{-d}=\frac{\mathfrak e}{4}$, and the pseudohyperbolic disks $D_{\frac{\mathfrak e}{4}}(z_n)$, $n\in\N$, are mutually disjoint because $\zeta$ is an $\mathfrak e$-chain. Moreover, by our construction, see \eqref{e3.5}, \eqref{e3.13},  the boundary circle of each pseudohyperbolic disk $D_{r}(z_n)$
 contains $N_\varepsilon:=\lfloor\log_2\frac{1}{\varepsilon}\rfloor +1$ points of zero locus $\zeta_\varepsilon$ of $B_{\zeta_\varepsilon}$ equally spaced with respect to $\rho$, and also $\bar\zeta_\varepsilon\cap \bar K=\emptyset$.  Thus, it remains to prove inequality \eqref{eq1.8} for the characteristic $\delta(\zeta_\varepsilon)$.
 
In the proof, we use the following technical result.
\begin{Lm}\label{lem3.3}
If $z\in\Di$ and $w\in \overline{D(z,tR)}$ for some 
 $R,t\in (0,1)$, then
 for each $y\in\Di\setminus D(z,R)$
\begin{equation}\label{e3.18}
\rho(y,z)^{a_{R,t}}\le\rho(y,w)\le \rho(y,z)^{b_{R,t}},
\end{equation}
where $a_{R,t}:=\frac{\ln\frac{1-tR^2}{(1-t)R}}{\ln \frac{1}{R}}$, $b_{R,t}:=\frac{\ln\frac{1+tR^2}{(1+t)R}}{\ln \frac{1}{R}}$.
\end{Lm}
\begin{proof}
The functions $u=\ln\rho(\cdot,w)$ and $v=\ln\rho(\cdot, z)$  are harmonic on $K:=\Di\setminus \overline{D(z,R)}$, equal to $0$  on the unit circle (outer boundary of $K$) and by virtue of the triangle inequality for $\rho$, see \cite[Ch.\,I,\,Lm.\,1.4]{Ga}, for $y$ on the boundary of  $ D(z_n,R)$ (inner boundary of $K$) we have
\[
b_{R,t}=\frac{\ln\frac{1+R\cdot tR}{R+tR}}{\ln \frac{1}{R}}\le \frac{\ln\frac{1+\rho(y,z)\cdot \rho(z,w)}{\rho(y,z)+\rho(z,w)}}{\ln \frac{1}{R}}\le \frac{u(y)}{v(y)}\le\frac{\ln\frac{1-\rho(y,z)\cdot \rho(z,w)}{\rho(y,z)-\rho(z,w)}}{\ln \frac{1}{R}}\le \frac{\ln\frac{1-R\cdot tR}{R-tR}}{\ln \frac{1}{R}}=a_{R,t}
\]
Hence, by the maximum modulus principle for harmonic functions
\[
a_{R,t} v \le u \le b_{R,t} v\quad {\rm on}\quad \Di\setminus D(z,R),
\]
as required.
\end{proof}

Since in our notation $\zeta=\{z_m\}_{m\in\N}$, each $z_m$ belongs to a uniquely defined sequence $\zeta_{i(m)}$ for some $1\le i(m)\le k$, and so there exists a uniquely defined $n(m)\in\N$ such that
$z_{m}=z_{n(m),i(m)}$. We set
\begin{equation}\label{e3.19}
z_{m,j,\varepsilon}:=z_{n(m),i(m),j,\varepsilon},\qquad m\in\N,\quad 1\le i\le k,\quad 1\le j\le N_\varepsilon.
\end{equation}
Then $\zeta_\varepsilon=\{z_{m,j,\varepsilon}\}_{m,j}$,
see \eqref{e3.5}, \eqref{e3.13}; hence, by  \eqref{eq1.2},
\begin{equation}\label{e3.20}
\begin{array}{l}
\displaystyle \delta(\zeta_{\varepsilon}):=\inf_{(m_1,j_1)}
\left\{\prod_{(m,j)\ne (m_1,j_1)}
\rho(z_{m_1,j_1,\varepsilon},z_{m,j,\varepsilon})\right\}\medskip\\
\displaystyle =\inf_{(m_1,j_1)}\left\{\left(\prod_{j=1}^{N_\varepsilon}\prod_{m\ne m_1}\rho(z_{m_1,j_1,\varepsilon},z_{m,j,\varepsilon})\right)\cdot\prod_{j\ne j_1}\rho(z_{m_1,j_1,\varepsilon},z_{m_1,j,\varepsilon})\right\}\medskip\\
\displaystyle =\inf_{(m_1,j_1)}\left\{\prod_{j=1}^{N_\varepsilon} \prod_{m\ne m_1}\rho(z_{m_1,j_1,\varepsilon},z_{m,j,\varepsilon})\right\}\cdot N_\varepsilon  \left(\frac{\mathfrak e}{4}\right)^{N_\varepsilon -1}\frac{1-\left(\frac{\mathfrak e}{4}\right)^2}{1-\left(\frac{\mathfrak e}{4}\right)^{2N_\varepsilon }}.
\end{array}
\end{equation}
Here we used the fact that the points $z_{m_1,j,\varepsilon}$, $1\le j\le N_\varepsilon$, on the boundary of $D(z_{m_1},\frac{\mathfrak e}{4})$
are equally spaced with respect to $\rho$, so that
\[
\prod_{j\ne j_1}\rho(z_{m_1,j_1,\varepsilon},z_{m_1,j,\varepsilon})=\prod_{j=1}^{N_\varepsilon-1}\rho\bigl(\mbox{$\frac{\mathfrak e}{4},\frac{\mathfrak e}{4}\cdot e^{\frac{2\pi\sqrt{-1}\cdot j}{N_{\varepsilon}}}$}\bigr)=N_\varepsilon  \left(\frac{\mathfrak e}{4}\right)^{N_\varepsilon -1}\frac{1-\left(\frac{\mathfrak e}{4}\right)^2}{1-\left(\frac{\mathfrak e}{4}\right)^{2N_\varepsilon }}.
\]

Next, since $\zeta$ is an $\mathfrak e$-chain, $\rho(z_{m_1},z_m)\ge \mathfrak e$ for $m\ne m_1$. Hence, by the triangle inequality,
\[
\rho(z_{m_1}, z_{m,j,\varepsilon})\ge  \rho(z_{m_1}, z_{m})-\rho(z_{m}, z_{m,j,\varepsilon})\ge  \mathfrak e-\frac{\mathfrak e}{4}=\frac{3\mathfrak e}{4}.
\]
Applying Lemma \ref{lem3.3} with $z:=z_{m_1}$, $w:=z_{m_1,j_1,\varepsilon}$, $R:=\frac{3\mathfrak e}{4}$ and $t=\frac 13$ we get from \eqref{e3.18} for $y=z_{m,j,\varepsilon}\in \Di\setminus D(z,R)$
\begin{equation}\label{e3.21}
\rho(z_{m,j,\varepsilon},z_{m_1})^{a_1}\le\rho(z_{m,j,\varepsilon},z_{m_1,j_1,\varepsilon})\le \rho(z_{m,j,\varepsilon},z_{m_1})^{b_1},
\end{equation}
where $a_1:=a_{\frac{3\mathfrak e}{4},\frac 13 }:=\frac{\ln\frac{16-3\mathfrak e^2}{8\mathfrak e}}{\ln \frac{4}{3\mathfrak e}}$,
$b_1:=b_{\frac{3\mathfrak e}{4},\frac 13}:=\frac{\ln\frac{16+ 3\mathfrak e^2}{16\mathfrak e}}{\ln \frac{4}{3\mathfrak e}}$.

Similarly, applying Lemma \ref{lem3.3} with $z:=z_{m}$, $w:=z_{m,j,\varepsilon}$, $R:=\mathfrak e$ and $t=\frac 14$ we get from \eqref{e3.18} for $y=z_{m_1}\in \Di\setminus D(z,R)$
\begin{equation}\label{e3.22}
\rho(z_{m_1},z_m)^{a_2}\le\rho(z_{m_1},z_{m,j,\varepsilon})\le \rho(z_{m_1},z_{m})^{b_2},
\end{equation}
where $a_2:=a_{\mathfrak e ,\frac 14 }:=\frac{\ln\frac{4- {\mathfrak e}^2}{3 \mathfrak e}}{\ln \frac{1}{\mathfrak e}}$,
$b_2:=b_{\mathfrak e ,\frac 14}:=\frac{\ln\frac{4+  \mathfrak e^2}{5 \mathfrak e}}{\ln \frac{1}{\mathfrak e}}$.
Now, combining \eqref{e3.21} and \eqref{e3.22} we obtain 
\begin{equation}\label{e3.23}
\rho(z_{m_1},z_m)^{a}\le\rho(z_{m,j,\varepsilon},z_{m_1,j_1,\varepsilon})\le \rho(z_{m_1},z_m)^{b},
\end{equation}
where $a:=a_1a_2$, $b:=b_1b_2$.

Let us estimate the term in  brackets in the last line of \eqref{e3.20} using \eqref{e3.23}. Then we have
\[
 \delta(\zeta_{\varepsilon})\le \left(\inf_{m_1}\prod_{m\ne m_1}\rho(z_{m_1},z_m)\right)^{bN_\varepsilon}\cdot N_\varepsilon  \left(\frac{\mathfrak e}{4}\right)^{N_\varepsilon -1}\frac{1-\left(\frac{\mathfrak e}{4}\right)^2}{1-\left(\frac{\mathfrak e}{4}\right)^{2N_\varepsilon }},
\]
and
\[
\delta(\zeta_{\varepsilon})\ge\left(\inf_{m_1}\prod_{m\ne m_1}\rho(z_{m_1},z_m)\right)^{aN_\varepsilon}\cdot N_\varepsilon  \left(\frac{\mathfrak e}{4}\right)^{N_\varepsilon -1}\frac{1-\left(\frac{\mathfrak e}{4}\right)^2}{1-\left(\frac{\mathfrak e}{4}\right)^{2N_\varepsilon }}.
\]
Or, equivalently (using that $r:=\frac{\mathfrak e}{4}$),
\begin{equation}\label{equ3.25}
(\delta(\zeta))^{aN_\varepsilon}\cdot N_\varepsilon  r^{N_\varepsilon -1}\frac{1-r^2}{1-r^{2N_\varepsilon }}\le  \delta(\zeta_{\varepsilon})\le (\delta(\zeta))^{bN_\varepsilon}\cdot N_\varepsilon  r^{N_\varepsilon -1}\frac{1-r^2}{1-r^{2N_\varepsilon }}.
\end{equation}
By definition, the function $(0,1)\ni\varepsilon\mapsto\delta(\zeta_\varepsilon)\in (0,1]$ is constant on the intervals $[2^{-m},2^{-m+1})$, $m\in\N$, and for each $\varepsilon\in \{2^{-m}\, :\, m\in\N\}$, 
\[
N_\varepsilon=-\log_2\varepsilon +1.
\]
Hence,
\[
(\delta(\zeta))^{aN_\varepsilon}r^{N_\varepsilon -1}= (\delta(\zeta))^{a}\varepsilon^{\log_2\frac{1}{r(\delta(\zeta))^a}}\quad {\rm and}\quad  (\delta(\zeta))^{bN_\varepsilon}r^{N_\varepsilon -1}=(\delta(\zeta))^{b}\varepsilon^{\log_2\frac{1}{r(\delta(\zeta))^b}}.
\]
Using these and \eqref{equ3.25} we get for such $\varepsilon$:
\[
(\delta(\zeta))^{a}\frac{1-r^2}{1-r^{2N_\varepsilon }}N_\varepsilon\varepsilon^{\log_2\frac{1}{r(\delta(\zeta))^a}} <\delta(\zeta_\varepsilon)< (\delta(\zeta))^{b}\frac{1-r^2}{1-r^{2N_\varepsilon }}N_\varepsilon\varepsilon^{b\log_2\frac{1}{r(\delta(\zeta))^b }}.
\]
It is easy to check using the explicit formulas for $a$ and $b$ that 
$a=a(r)\in (1,2 )$, $b=b(r)\in (0,1)$ and $a-1=O\bigl(\frac{1}{\ln\frac{1}{r}}\bigr)$, $1-b=O\bigl(\frac{1}{\ln\frac{1}{r}}\bigr)$ as $r\to 0^+$, as required.

This completes the proof of part (iii) and Theorem \ref{te1.1}(1).
\sect{Proofs of Theorem \ref{te1.1}(2) and \ref{te1.1}(3)}
\subsect{Proof of Theorem \ref{te1.1}(2)}
We proceed along the lines of the proof of Theorem 1.7 of \cite{Br1}.
Let $\widehat K$ be the minimal holomorphically convex subset of $\mathfrak M$ containing $K$. Then by the hypothesis of the theorem, $\widehat K\subset\widehat U$.
Using the holomorphic convexity of $\widehat K$ we obtain:
\begin{Lm}[\mbox{\cite[Lm.\,5.1]{Br1}}]\label{le4.1}
There exist functions $f_1,\dots, f_k\in H^\infty$ such that
\begin{equation}\label{e4.1}
\widehat K\subset \Pi:=\bigl\{x\in\mathfrak M\, :\, \max_{1\le i\le k} |\hat f_i(x)|\le 1\bigr\}\subset\widehat U.
\end{equation}
\end{Lm}
We set $\widehat U':=\mathfrak M\setminus \Pi$. Applying the argument in the proof of part (1) of the theorem, see section~3.1, to the cover $(\widehat U,\widehat U')$ of $\mathfrak M$, we obtain an open cover $(\widehat V,\widehat V')$ of $\mathfrak M$ such that $\Pi\subset\widehat V$, and a partition of unity $\{\rho_{ V},\rho_{ V'}\}$ subordinate to the open cover 
$(V, V')$ of $\Di$, where $V=\widehat V\cap\Di$ and  $V'=\widehat V'\cap\Di$, satisfying 
\begin{equation}\label{equa4.2}
\widehat V\cap \widehat V'\Subset\mathfrak M_a\quad \textrm {and }\quad
\frac{\partial\rho_{ V}}{\partial\bar z}=-\frac{\partial \rho_{ V'}}{\partial\bar z}=\frac{g(z)}{1-|z|^2},\ \, z\in\Di,  \textrm{ \ where\ } g\in C_\rho^\infty(\Di).
\end{equation}
Moreover, the closure of the support of $g$ in $\mathfrak M$, $\overline{{\rm supp}(g)}$,  is a compact subset of $\widehat V\cap \widehat V'$.  Thus, $\overline{{\rm supp}(g)}\cap\Pi=\emptyset$. This implies that the function $\max_{1\le j\le k}|\hat f_{j}|$ is greater than $1$ on $\overline{{\rm supp}(g)}$. By  \eqref{equa4.2}, $\widehat W:=\widehat V\cap \widehat V'\Subset\mathfrak M_a$ is an open neighbourhood of $\overline{{\rm supp}(g)}$ which does not intersect $\Pi$. Thus for some $r>1$ the set $\overline{{\rm supp}(g)}$ is covered by open sets 
\[
\widehat W_i:=\{x\in \mathfrak M\,:\, |\hat f_{i}(x)|>r\}\cap \widehat W,\quad 1\le i\le k. 
\]
Let $\zeta=\{z_n\}$ be a $\frac 12 $-chain of ${\rm supp}(g)$ with respect to $\rho$. By the argument of the proof of Proposition \ref{prop2.4}, $\zeta $ is an interpolating sequence for $H^\infty$. 

\noindent Let
\[
O:=\bigsqcup_{n}D(z_{n},\mbox{$\frac 12 $})
\]
be the pseudohyperbolic $\frac 12$-neighbourhood of $\zeta$. By  definition, $O$ is an open neighbourhood of ${\rm supp}(g)$ and $\overline{O}\subset\mathfrak M_a$.

We set 
\begin{equation}\label{e4.2}
U_i:=\widehat W_i\cap O\quad {\rm and }\quad U_i ':=U_i\setminus\cup_{j=1}^{i-1}U_j,\quad 1\le i\le k.
\end{equation}
Then $\overline{U}_i\subset\mathfrak M_a$; hence, by Proposition \ref{prop2.4} each $U_i$ is a quasi-interpolating set. In particular, the bounded linear operators
$L_{U_i}^X: L^\infty(U_i,X)\to C_\rho(\Di,X)$, $1\le i\le k$,  of Theorem \ref{te2.1} are well defined. 

Let $\varepsilon\in (0,1)$ and let
 \begin{equation}\label{e4.3}
n_\varepsilon:=\left\lfloor\log_r\left(\frac{1}{\varepsilon}\right)\right\rfloor+1\in \N.
\end{equation}
Let us consider the family of bounded linear operators $E_{\varepsilon}^X: H^\infty(U,X)\to C_\rho(\Di,X)$,
\begin{equation}\label{e4.4}
E_{\varepsilon}^X(f):=\rho_V f-\sum_{i=1}^k f_i^{n_\varepsilon}\cdot L_{U_i}^X\left(\frac{\chi_i\,g f}{f_i^{n_\varepsilon}}\right).
\end{equation}
Then the desired result is the consequence of the following:
\begin{Prop}\label{prop4.4}
Every  $E_{\varepsilon}^X$ maps $H^\infty(U,X)$ to $H^\infty(\Di,X)$
and satisfies 
\begin{itemize}
\item[(i)] 
For all $f\in H^\infty(U,X)$,
\begin{equation}\label{e4.5}
\sup_{z\in K}\left\|f(z)- (E_{\varepsilon}^X(f))(z)\right\|_X\le C_0\varepsilon\|f\|_{H^\infty(U,X)},
\end{equation}
where $C_0=k M\|g\|_{C_\rho(\Di)}$, $M:=\max_{1\le i\le k}\|L_{U_i}^X\|$.\smallskip
 \item[(ii)] 
 \begin{equation}\label{e4.6}
 \|E_{\varepsilon}^X\|\le C_1\varepsilon^{-d_1},
 \end{equation}
 where $C_1:=\frac{2R k M\|g\|_{C_\rho(\Di)}}{r}$,   $d_1:=\log_r \frac{R}{r}$,
 $R:=\max_{1\le i\le k}\|f_i\|_{H^\infty}$.
 \end{itemize}
 \end{Prop}
 
 \begin{proof}
Since $\frac{\chi_i\, gf}{f_i^{n_\varepsilon}}\in L^\infty(U_i,X)$  for all $f\in H^\infty(U,X)$, the operator $E_\varepsilon^X$ is well-defined.
Then by the definition of operators $L_{U_i}^X$, see Theorem \ref{te2.1}, we have (in the weak sense) for all $f\in H^\infty(U,X)$, see \eqref{e3.1},
\[
\frac{\partial\bigl(E_\varepsilon^X(f)\bigr)}{\partial\bar z}=\frac{\partial\bigl(\rho_V    f  \bigr)}{\partial\bar z}
- \sum_{i=1}^k f_i^{n_\varepsilon}\cdot\frac{\frac{\chi_i \, g  f}{f_i^{n_\varepsilon}}}{1-|z|^2}=\frac{\partial\bigl(\rho_V    f  \bigr)}{\partial\bar z}
- \frac{ g  f}{1-|z|^2}=\frac{\partial\bigl(\rho_V    f  \bigr)}{\partial\bar z}
- \frac{\partial\bigl(\rho_V    f  \bigr)}{\partial\bar z}=0.
\] 
Here we used  that $\sum_{i=1}^k\chi_i g=g$ because ${\rm supp}(g)\subset \cup_{i=1}^k U_i=\cup_{i=1}^k U_i'$.

This shows that every  $E_\varepsilon^X$ maps $H^\infty(U,X)$ to $H^\infty(\Di,X)$.

Now, for each $f\in H^\infty(U,X)$, we obtain in notation of equations \eqref{e4.5}, \eqref{e4.6},
\[
\begin{split}
\displaystyle
\|E_\varepsilon^X(f)\|_{H^\infty(\Di,X)}\le & \|\rho_V f\|_{C_\rho(\Di,X)}+\left(\max_{1\le i\le k}\|f_i\|_{H^\infty}\!\right)^{n_\varepsilon}\cdot\sum_{i=1}^k\|L_{U_i}^X\|\cdot \left\|\frac{gf}{f_i^{n_\varepsilon}}\right\|_{L^\infty(U_i,X)}\smallskip\\
 \le &\|f\|_{H^\infty(U,X)}+ \left(\frac{R}{r}\right)^{n_\varepsilon}k M\|g\|_{C_\rho(\Di)}\|f\|_{H^\infty(U,X)}.
\end{split}
\]

From this and the definition of $n_\varepsilon$, see \eqref{e4.3},  follows \eqref{e4.6}:
\[
\|E_\varepsilon^X\|\le \left(1+\varepsilon^{-\log_r \frac{R}{r}}\right)\frac{R k M\|g\|_{C_\rho(\Di)}}{r}\le C_1\varepsilon^{-d_1}.
\]

Similarly, since $K\subset\Pi\cap\Di\subset V$ and $\max_{1\le i\le k}|\hat f_i|_\Pi|\le 1$, we obtain \eqref{e4.5}:
\[
\begin{split} 
 \sup_{z\in K}\left\|f(z)- (E_\varepsilon^X(f))(z)\right\|_X= &\sup_{z\in K}\left\|\sum_{i=1}^k f_i^{n_\varepsilon}(z)\cdot L_{U_i}^X\left(\frac{\chi_i\,g f}{f_i^{n_\varepsilon}} \right)\!(z)\right\|_X\smallskip\\
\le &\sum_{i=1}^k\|L_{U_i}^X\|\cdot \left\|\frac{gf}{f_i^{n_\varepsilon}}\right\|_{L^\infty(U_i,X)}\le \frac{k M\|g\|_{C_\rho(\Di)}\|f\|_{H^\infty(U,X)}}{r^{n_\varepsilon}}\smallskip\\
\le &\varepsilon k M\|g\|_{C_\rho(\Di)}\|f\|_{H^\infty(U,X)}=:\varepsilon C_0 \|f\|_{H^\infty(U,X)}.
\end{split}
\]
This completes the proof of the proposition and Theorem \ref{te1.1}(2).
\end{proof}
\subsect{Proof of Theorem \ref{te1.1}(3)}
The statement follows from a similar statement for the operator of Theorem \ref{te2.1} and our definitions of the operators $L_\varepsilon^X$ and $E_\varepsilon^X$, see \eqref{e3.14} and \eqref{e4.3}, respectively.

The proof of Theorem \ref{te1.1} is complete.

\sect{Proofs of Theorem \ref{cor1.3} and Corollary \ref{cor1.2}}
\subsect{Proof of Theorem \ref{cor1.3}}
(i) If $K=K_1\times\cdots\times K_n\subset\Di^n$, then its closure in $\mathfrak M^n$ is $\bar K=\bar K_1\times\cdots\times \bar K_n$. Therefore, reducing the neighbourhood $\widehat U$ of $\bar K$, if necessary, without loss of generality, we can assume that it is of the form  $\widehat U=\widehat U_1\times\cdots\times \widehat U_n$, where $\widehat U_i\subset\mathfrak M$ is an open neighbourhood of $\bar K_i$, $1\le i\le n$. Hence, if $U_i:=\widehat U_i\cap\Di$, then
$U=\widehat U\cap\Di^n=U_1\times\cdots \times U_n$. We prove the result by induction on $n$.

For $n=1$ the required statement follows from Theorem \ref{te1.1}\,(i). Assuming that the statement is proved for all $k\le n-1$ with $n>1$ let us prove it for $k=n$. 

Let $K^{n-1}:=K_1\times\dots \times K_{n-1}\subset\Di^{n-1}$, $\widehat U^{n-1}:=\widehat U_1\times\dots\times \widehat U_{n-1}\subset\mathfrak M^{n-1}$, and  $U^{n-1}=\widehat U^{n-1}\cap\Di^{n-1}$. 
Let us consider the function $f\in H^\infty(U,X)$ as an element of $H^\infty(U^{n-1},H^\infty(U_n,X))$ whose value at a point
$(z_1,\dots, z_{n-1})\in U^{n-1}$ is the function $f(z_1,\dots, z_{n-1},\cdot)\in H^\infty(U_n,X)$. Then, applying the induction hypothesis to the function space $H^\infty(U^{n-1},H^\infty(U_n,X))$ we obtain that given $\varepsilon>0$ there exists $h'\in H^\infty(\Di^{n-1},H^\infty(U_n,X))$  and
 $b'=b_1\cdots b_{n-1}\in H^\infty(\Di^{n-1})$, where $b_i$ is an interpolating Blaschke product in $z_i$, $1\le i\le n-1$, and $\inf_{K^{n-1}}|b'|=m>0$, such that
 \begin{equation}\label{eq5.1}
 \sup_{(z_1,\dots, z_{n-1})\in K^{n-1}}\left\|f(z_1,\dots, z_{n-1},\cdot)-\frac{h'(z_1,\dots , z_{n-1})}{b'(z_1,\dots, z_{n-1})}\right\|_{H^\infty(U_n,X)}<\frac{\varepsilon}{2}.
 \end{equation}
 Let us consider the function $h'\in H^\infty(\Di^{n-1},H^\infty(U_n,X))$ as an element of
 the space $H^\infty(U_n, H^\infty(\Di^{n-1},X))$ whose value at $z_n\in U_n$ is the function $(h'(\cdot))(z_n)\in H^\infty(\Di^{n-1},X)$. Applying the induction hypothesis to the function space $H^\infty(U_n, H^\infty(\Di^{n-1},X))$ we obtain that there exist
 $h''\in H^\infty(\Di, H^\infty(\Di^{n-1},X))$ and an interpolating Blaschke product $b_n$ such that $\inf_{K_n} |b_n|>0$ and
 \begin{equation}\label{eq5.2}
 \sup_{z_n\in K_n}\left\|(h'(\cdot))(z_n)- \frac{h''(z_n)}{b_n(z_n)}\right\|_{H^\infty(\Di^{n-1},X)}<\frac{\varepsilon m}{2}.
 \end{equation}
 
We set for $(z_1,\dots, z_n)\in\Di^n$,
 \[
 h(z_1,\dots, z_n):=(h''(z_n))(z_1,\dots, z_{n-1}), \quad b(z_1,\dots, z_n):=b'(z_1,\dots, z_{n-1})\cdot b_n(z_n).
 \]
 Clearly, $h\in H^\infty(\Di^n,X)$  and $\inf_K |b|>0$. Combining \eqref{eq5.1} and \eqref{eq5.2} we obtain for $z=(z^{n-1},z_n)\in K$, where $z^{n-1}:=(z_1,\dots, z_{n-1})\in K^{n-1}$,
 \[
 \left\|f(z)-\frac{h(z)}{b(z)} \right\|_X\le  \left\|f(z)-\frac{h'(z^{n-1})(z_n)}{b'(z^{n-1})}\right\|_{X}
 + \left\|\frac{h'(z^{n-1})(z_n)- \frac{h''(z_n)}{b_n(z_n)}}{b'(z^{n-1})}\right\|_{X}<\frac{\varepsilon}{2}+\frac{\varepsilon m}{2m}=\varepsilon.
 \]
This completes the proof of the induction step and hence of part (i) of the theorem. \smallskip

\noindent (ii) The proof repeats the above. Here, instead of $\bar K$ we work with a compact set $\widehat K\subset\widehat U$ containing $K$ which is an $n$-ary Cartesian product of holomorphically convex subsets of $\mathfrak M$ (existing by the hypothesis of the theorem), and
use part (ii) of Theorem \ref{te1.1} instead of part (i). We leave the details to the reader.
\smallskip

\noindent (iii) Let us prove that if $f$ admits a continuous extension to $\widehat U$, then the function $h\in H^\infty(\Di^n,X)$ in part (i) above can be chosen so that it admits a continuous extension  to $\mathfrak M^n$. 

To this end, in the above notation, reducing the neighbourhood $\widehat U=\widehat U_1\times\cdots\times \widehat U_n$ of $\bar K$, if necessary, without loss of generality we can assume that $f\in H^\infty(U,X)$ admits a continuous extension to
the closure $\bar U=\bar U_1\times\cdots\times \bar U_n\subset\mathfrak M^n$ of $\widehat U$. Then we repeat the induction argument of the above proof. 
Here, at the first step of induction by the hypotheses we obtain that $f\in H_{\rm comp}^\infty (U,X)$ and thus the required result follows from part (iii) of Theorem \ref{te1.1}.
Next, in the notation of the proof of part (i) above, in the proof of the induction step we consider $f\in H^\infty(U,X)$ as a function of $H^\infty(U^{n-1},H^\infty(U_n,X))$. Since $f\in H^\infty(U,X)$ admits a continuous extension to the closure $\bar U$, by the Arzela-Ascoli theorem,  this function also lies in  $H^\infty(U^{n-1},H^\infty(U_n,X)\cap (C(\bar U_n,X)|_{U_n}))$ and admits a continuous extension to $\widehat U^{n-1}$. 
Here $C(\bar U_n,X)|_{U_n}$ is the Banach space of bounded $X$-valued continuous functions on $U_n$  that admit continuous extensions to $\bar U_n$ equipped with supremum norm.
Thus by the induction hypothesis there exist a function $h'\in H^\infty(\Di^{n-1},H^\infty(U_n,X)\cap(C(\bar U_n,X)|_{U_n}))$  which admits a continuous extension to $\mathfrak M^{n-1}$, and a function $b'\in H^\infty(\Di^{n-1})$ as in (i) above, satisfying equation \eqref{eq5.1}.  It follows that $h'$ can be viewed as a function of
 $H^\infty(\Di^{n-1}\times U_n,X)$ which admits a continuous extension  to
 $\mathfrak M^{n-1}\times\bar U_n$. Therefore, it can also be considered as  a function of $H^\infty(U_n, H^\infty(\Di^{n-1},X)\cap (C(\mathfrak M^{n-1},X)|_{\Di^{n-1}}))$ which admits a continuous extension to $\widehat U_n$. Then, applying the induction hypothesis to the latter  space, we obtain a  function $h''\in H_{\rm comp}^\infty(\Di, H^\infty(\Di^{n-1},X)\cap (C(\mathfrak M^{n-1},X)|_{\Di^{n-1}}))$ and an interpolating Blaschke product $b_n$ satisfying equation \eqref{eq5.2}.  By definition, $h''$ determines a function of $H^\infty(\Di^n,X)$ that admits a continuous extension to $\mathfrak M^{n}$.
 The end of the proof is the same as in part (i) above. 
 
 The proof that if $f$ admits a continuous extension to $\widehat U$, then the function $h\in H^\infty(\Di^n,X)$ in part (ii) above can be chosen such that it  admits a continuous extension  to $\mathfrak M^n$ is similar and can be left to the reader.
 
 The proof of Theorem \ref{cor1.3} is complete.
 \subsect{Proof of Corollary \ref{cor1.2}}
 Since by Su\'arez's theorem \cite[Thm.\,2.4]{S0} the algebra $H^\infty$ is separating, 
 for an open set $\widehat O\subset\mathfrak M$ and a point 
 $x\in\widehat O$ there exists a polynomially convex set $\widehat K_x\subset\widehat O$ of the form $\{y\in\mathfrak M\, :\, |\hat f(y)|\le 1\}$ for some $f\in H^\infty$, such that its interior $\widehat U_x$ contains $x$.
 This implies that  an open set $\widehat U\subset\mathfrak M^n$ admits an open cover $(\widehat U_i)_{i\in I}$, where each $\widehat U_i$ is the interior
of an $n$-ary product of polynomially convex subsets of $\mathfrak M$, $\widehat K_i\subset \widehat U$. We set $K_i:=\widehat K_i\cap\Di^n$ and $U_i:=\widehat U_i\cap\Di^n$.  

Suppose that $f$ is an $X$-valued holomorphic function on $U=\widehat U\cap\Di^n$ such that $f(V)\Subset X$ for every  subset $V\Subset\widehat U$ of $U$. In particular, this is valid for subsets $K_i$, $i\in I$. Then
Theorem \ref{cor1.3}\,(ii) implies that for each $i\in I$  there exists a sequence of functions $\{f_{ij}\}_{j\in\N}\subset H^\infty(\Di^n,X)$ which converges uniformly on $K_i$ to $f|_{K_i}$. Further, each $g\in H^\infty(\Di^n,X)$ extends (by means of the Gelfand transform) to a weak-$*$ continuous function $\hat g$ on $\mathfrak M(H^\infty(\Di^n))$ with values in the bidual $X^{**}$ of $X$ such that for every open set $\widetilde W\subset {\rm cl}(\Di^n)$, and $W:=\widetilde W\cap\Di^n\subset\Di^n$,
 \[
 \sup_{x\in\widetilde W}\|\hat g(x)\|_{X^{**}}=\|g\|_{H^\infty(W,X)}.
 \] 
 This implies that the sequence $\{\hat f_{ij}|_{\widetilde U_i}\}_{j\in\N}$, where $\widetilde U_i:=\pi_n^{-1}(\widehat U_i)$, $i\in I$, converges uniformly to an $X^{**}$-valued weak-$*$ continuous function  $\hat f_i$ on $\widetilde U_i$ which extends $f|_{U_i}$. (Note that $\pi_n^{-1}(U_i)=U_i$ and so $U_i$ is dense in $\widetilde U_i$ as $U_i$ is dense in $\widehat U_i\subset\mathfrak M^n$.)
  Since $f(U_i)\subset f(K_i)\Subset X\subset X^{**}$ (here we naturally identify $X$ with a closed subspace of $X^{**}$), and each compact subset of $X$ is weak-$*$ compact in $X^{**}$, the image of the function $\hat f_i$ belongs to a compact subset of $X$, the closure in $X$ of $f(K_i)$. But the weak-$*$ and strong (norm) topologies are equivalent on every (strong)  compact subset of $X$; hence, $\hat f_i$ is an $X$-valued continuous function on $\widetilde U_i$ which extends $f|_{U_i}$. Finally, if $\widetilde U_i\cap\widetilde U_j$ for some $i,j\in I$, then $U_i\cap U_j$ is a dense open subset of 
 $\widetilde U_i\cap\widetilde U_j$ so that $\hat f_i=\hat f_j$ on $\widetilde U_i\cap\widetilde U_j$ as these functions are continuous extensions of the same function $f$. Therefore, the continuous function $\hat f\in C(\widetilde U,X)$, where $\widetilde U:=\pi_n^{-1}(\widehat U)$, given by $\hat f(x)=\hat f_i(x)$ if $x\in\widetilde U_i$, $i\in I$, is well defined and extends the function $f$, as required.
 
 Conversely, it is obvious that if an  $X$-valued holomorphic function $f$ on $U=\widehat U\cap\Di^n$ admits a continuous extension to the open set $\pi_n^{-1}(\widehat U)\subset {\rm cl}(\Di^n)$, then  $f(V)\Subset X$ for  every subset $V\Subset\widehat U$ of $U$.
 
 The proof of the corollary is complete.

\end{document}